\documentclass[secthm,seceqn,amsthm,ussrhead,12pt]{amsart}
\usepackage{amsmath,latexsym}
\usepackage[english]{babel}
\usepackage[psamsfonts]{amssymb}
\usepackage{times}
\usepackage{cite}

\usepackage[mathcal]{euscript}
\numberwithin{equation}{section} \textwidth=17.5cm
\newtheorem{theorem}{Theorem}[section]
\newtheorem{lemma}[theorem]{Lemma}
\newtheorem{example}[theorem]{Example}

\newtheorem{definition}[theorem]{Definition}

\numberwithin{equation}{section}

\begin{document}

\title{Geometric Description of L$_1$-Spaces}

 \author[M. M. Ibragimov]{M. M. Ibragimov}
 \author[K. K. Kudaybergenov]{K. K. Kudaybergenov}

\address{Department of Mathematics,
 Karakalpak state university
Ch. Abdirov 1
  230113, Nukus,    Uzbekistan,}

\email{mukhtar$_{-}$nukus@mail.ru}

\email{karim2006@mail.ru}

\begin{abstract}
We describe strongly facially symmetric spaces which are
isometrically isomorphic to L$_1$-space.
\end{abstract}


\maketitle
\section{Introduction}

An important problem in the theory of operator algebras is a
geometric characterization of state spaces of operator algebras.
In this connection, in the mid-1980s, Y. Friedman and B. Russo
introduced facially symmetric spaces \cite{fr1}. The main goal of
studying these spaces was to obtain a geometric characterization
of preduals of $JBW^\ast$-triples admitting an algebraic
structure. Many properties required in these characterizations
represent natural assumptions concerning state spaces of physical
systems. Such spaces are referred to as geometric models for
states in the quantum mechanics. In \cite{fr3} it is proved that
preduals of von Neumann algebras and, more generally, those of
$JBW^\ast$-triples are neutral strongly facially symmetric spaces.

An attempt to classify facially symmetric spaces was done in
\cite{fr4}, where one has proposed a geometric characterization of
complex Hilbert spaces and complex spin factors and described
$JBW^\ast$-triples of ranks 1 and 2 and Cartan factors of types 1
and 4. Later Y. Friedman and B. Russo \cite{fr5} have described
atomic facially symmetric spaces. Namely, they have proved that a
neutral strongly facially symmetric space is linearly isometric to
the predual of one of Cartan factors of types 1-6, provided that
it satisfies four natural and physically significant axioms which
are known to hold in preduals of all $JBW^\ast$-triples. M. Neal
and B. Russo \cite{NR} have established geometric conditions
ensuring that a facially symmetric space is isometric to the
predual space of a $JBW^\ast$-triple. In particular, they have
proved that any neutral strongly facially symmetric space is
decomposable into the direct sumof atomic and nonatomic strongly
facially symmetric spaces. In \cite{MZ} one has fully described
strongly facially symmetric spaces isometrically isomorphic to
preduals of atomic commutative von Neumann algebras.

In this paper we describe strongly facially symmetric spaces which
are isometrically isomorphic to L$_1$-space.

\section{PRELIMINARIES}

  Let $Z$ be a real or complex normed space. In what follows, two elements  $x, y \in Z$
  are said to be
\textit{orthogonal} (we write $x \diamondsuit y,$) if  $\|x +
y\|=\|x - y\|=\|x\|+\|y\|.$ We also say that subsets  $S, T\subset
Z$ are orthogonal (and write  $(S \diamondsuit T)$), if  $x
\diamondsuit y$ for all pairs  $(x, y)\in S\times T.$ We put
$S^\diamondsuit =\{x \in Z: x \diamondsuit y,\, \forall\,  y\in
S\}$  for a subset $S$ in the space $Z.$ The set $S^\diamondsuit$
is called the \textit{orthogonal complement} of $S.$ A convex
subset $F$ of the unit ball $Z_1=\{x\in Z:\|x\|\leq 1\}$ is called
a \textit{face}, if the condition   $\lambda y+(1-\lambda)z\in F,$
where
 $y, z\in Z_1$ and $\lambda
\in (0,1),$ implies  $y, z\in F.$ A face
 $F$ of the unit ball is said to be \textit{norm exposed}, if
  $F=F_u=\{x\in Z: u(x)=1\}$
for some  $u\in Z^*$ with  $\|u\| =1.$ An element  $u\in Z^\ast$
is called a \textit{projective unit}, if $\|u\| =1$ and $u(y)=0$
for all $y\in F_u^\diamondsuit$ (see \cite {fr1}).

\begin{definition}  \cite {fr1}. A norm exposed face $F_{u}$ of $Z_1$
is said to be \textit{symmetric}, if there exists a linear
isometry $S_u$ from $Z$ onto $Z$ such that $S_u^2=I$ and its fixed
point set exactly coincides with the topological direct sum of the
closure $\overline{sp}F_{u}$ of the linear hull of the face
$F_{u}$ and its orthogonal complement $F_{u}^{\diamond},$ i.e.,
coincides with $(\overline{sp}F_{u})\oplus F_{u}^\diamond.$
\end{definition}

\begin{definition}
 \cite {fr1}. A space $Z$ is called a \textit{weakly facially symmetric space}
  (a WFS-space), if each
norm exposed face of $Z_1$ is symmetric.
\end{definition}

For each symmetric face $F_u$ we define contractive projectors
 $P_k(F_u),$ $k = 0, 1, 2$ on  $Z$ as follows.
Firstly, $P_1(F_u) = (I - S_u)/2$ is the projector onto the
eigenspace corresponding to the eigenvalue $-1$ of the symmetry
$S_u.$ Secondly, $P_2(F_u)$ and  $P_0(F_u)$ are defined as
projectors from  $Z$ onto $\overline{sp}F_u$ and
$F_u^\diamondsuit,$ respectively, i.e., $P_2(F_u) + P_0(F_u) = (I
+ S_u )/2.$ Projectors  $P_k(F_u)Т$s are said to be
\textit{geometric Pierce} projectors.

A projective unit  $u$ from  $Z^*$ is called a \textit{geometric
tripotent}, if  $F_u$ is a symmetric face and $S_u^*u=u$ for the
symmetry $S_u$ corresponding to  $F_u.$  We denote by
$\mathcal{GT}$ and $\mathcal{SF}$ sets of all geometric tripotents
and symmetric faces, respectively. The correspondence
$\mathcal{GT} \ni u \mapsto F_u \in \mathcal{SF}$ is bijective
(see  \cite[Proposition 1.6]{fr2}). For each geometric tripotent
$u$ from the dual of a WFS-space $Z,$  we denote Pierce projectors
by  $P_k(u) = P_k(F_u), k = 0, 1, 2.$ Further, we set $U= Z^\ast,
Z_k(u) = Z_k(F_u)= P_k(u)Z$ and $U_k(u) = U_k(F_u)=
P_k(u)^\ast(U).$ Therefore, we get the Pierce decomposition  $Z =
Z_2(u) + Z_1(u) + Z_0(u)$  and  $U = U_2(u) + U_1(u) + U_0(u).$
Tripotents  $u$  and  $v$  are said to be orthogonal, if $u \in
U_0(v)$ (which implies  $v \in U_0(u)$)  or, equivalently,
 $u \pm v \in
\mathcal{GT}$ (see   \cite[Lemma 2.5]{fr1}). More generally,
elements $a$ and   $b$ from  $U$  are said to be orthogonal, if
one of them belongs to $U_2(u)$  and the other one does to
$U_0(u)$  for some geometric tripotent  $u.$

A contractive projector $Q$ on  $Z$ is said to be
\textit{neutral}, if for each $x \in Z$ the equality $\|Q x \|=
\|x\|$ implies  $Q x=x.$ A space $Z$ is said to be neutral, if for
each $u\in \mathcal{GT}$ the projector $P_2(u)$  is neutral.

\begin{definition}\cite {fr1}.  A WFS-space $Z$
is said to be \textit{strongly facially symmetric} (an SFS-space),
if for every norm exposed face $F_u$ of $Z_1$ and for every $g \in
Z^\ast$ such that $\|g\| = 1$ and $F_u \subset F_g,$ the equality
$S_u^\ast g = g,$ is valid. Here  $S_u$ denotes the symmetry
associated with $F_u.$
\end{definition}

Instructive examples of neutral strongly facially symmetric spaces
are Hilbert spaces, preduals of von Neumann algebras or
$JBW^\ast$-algebras and, more generally, preduals of
$JBW^\ast$-triples. Moreover, geometric tripotents correspond to
nonzero partial isometries in von Neumann algebras and tripotents
in $JBW^\ast$-triples  (see \cite{fr3}).

In a neutral strongly facially symmetric space $Z$ each nonzero
element admits a polar decomposition
 \cite[Theorem 4.3]{fr2}: i.e., for  $0 \neq x \in Z$ there exists a unique geometric tripotent
   $v =v_x$ such
that   $\langle v, x\rangle  = \|x\|$ and  $\langle v, x^\diamond
\rangle = 0.$ For two elements  $x, y \in Z$ we have  $x \diamond
y,$ if and only if  $v_x \diamond v_y$ (see \cite[Corolarry 1.3(b)
and Lemma 2.1]{fr1}). The set of geometric tripotents is ordered
as follows. For given $u, v \in \mathcal{GT}$ one sets $u \leq v,$
if $F_u \subset F_v.$ Note that this definition is equivalent to
the following two conditions. The first one is the equality
$P_2(u)^*v = u.$ The second one states that either  $v - u$ equals
zero or the geometric tripotent is orthogonal to $u$ (see
\cite[Lemma 4.2]{fr2}).

\section{The main result}

Let   $Z$ be a real neutral strongly facially symmetric space and
let  $e\in Z^\ast$ be a geometric tripotent such that
\begin{equation}\label{yad}
Z_1=co(F_e\cup F_{-e}).
\end{equation}
We put
\[
\nabla = \{u \in \mathcal{GT}: u \le e\} \cup \{0\}.
\]
As is known \cite[Proposition 4.5]{fr2}, the set $\nabla$ is a
complete orthomodular lattice with the orthocomplementation
 $u^\perp = e - u$
with respect to the order $"\le".$

\begin{example} The space $\mathbb{R}^n$
with the norm
$$
||x||=\sum\limits_{i=1}^{n}|t_i|,\, x=(t_i)\in \mathbb{R}^n
$$
is an SFS-space. Let us consider the geometric tripotent
$$
e=(1, 1,\cdots, 1)\in \mathbb{R}^n\cong (\mathbb{R}^n)^\ast.
$$
The face
$$
F_e=\left\{x\in \mathbb{R}^n:
 \sum\limits_{i=1}^{n}t_i=1,\, t_i\geq 0, i=\overline{1, n}\right\}
$$
satisfies property \eqref{yad}. In this case we have
$$
\nabla=\left\{u=(\varepsilon_i): \varepsilon_i\in\{0, 1\},
i\in\overline{1, n}\right\}.
$$

More generally, let us consider a measurable space $(\Omega,
\Sigma, \mu)$ possessing the direct sum property and the space
$L_1(\Omega, \Sigma, \mu)$ of all equivalence classes of
integrable real-valued functions on $(\Omega, \Sigma, \mu).$ The
space $L_1(\Omega, \Sigma, \mu)$ with the norm
$$
||f||=\int\limits_\Omega|f(t)|\,d\mu(t),\, f\in L_1(\Omega,
\Sigma, \mu)
$$
is an SFS-space.

Consider the geometric tripotent $e\in L^\infty(\Omega, \Sigma,
\mu)\cong L_1(\Omega, \Sigma, \mu)^\ast,$ where  $e$ denotes the
class containing the function that identically equals $1.$ Then
the face
$$
F_e=\left\{f\in L_1(\Omega, \Sigma, \mu): f\geq
0,\,\int\limits_\Omega f(t)\,d\mu(t)=1\right\}
$$
satisfies property \eqref{yad}. In this case we have
$$
\nabla=\left\{\tilde{\chi}_A: A\in \Sigma\right\}
$$
where $\tilde{\chi}_A$ is the class containing the characteristic
function of the set $A\in \Sigma.$
\end{example}

Now let   $Z$ be a real neutral strongly facially symmetric space
such that there exists a geometric tripotent $e\in Z^\ast$
satisfying property~\eqref{yad}.

It is known \cite[Lemmata 1 and  и 2]{MZ} that if $\nabla$ is a
Boolean algebra, then for every $u\in \nabla,$ $u\neq 0$ the
following conditions are fulfilled:
\begin{enumerate}
\item   $P_1(u)=0;$

\item   $P_2 (u) = P_0 (u^\perp);$

\item $P_2(u+v)=P_2(u)+P_2(v),$ где
$u\diamondsuit v.$
\end{enumerate}

The following theorem is the main result of this paper.

\begin{theorem}\label{MTH}
Let  $Z$ be a real neutral strongly facially symmetric space and
let
  $e\in Z^\ast$ be a geometric
tripotent satisfying property \eqref{yad}.  If $\nabla$  is a
Boolean algebra, then there exists a measurable space $(\Omega,
\Sigma, \mu)$ possessing the direct sum property such that the
space $Z$ is isometrically isomorphic to the space  $L_1(\Omega,
\Sigma, \mu).$
\end{theorem}

To prove this theorem, we need ten lemmas.

Assume that $u, v\in \mathcal{GT}.$
 If $F_u\cap F_v\neq \emptyset,$
then we denote by $u\wedge v$ a geometric tripotent such that
  $F_{u\wedge v}=F_u\cap F_v;$
otherwise we put $u\wedge v =0.$

Let  $u, v\in \nabla.$ It is clear that  $u \diamondsuit v$
implies  $u\wedge v =0.$

Let $u\wedge v =0.$ Since $\nabla$ is a Boolean algebra, we get
$u\leq e-v.$ Therefore, in view of the condition $v \diamondsuit
(e- v)$ we have $u \diamondsuit v.$

Thus we get the following lemma.

 \begin{lemma}\label{dis}
Let  $u, v\in \nabla.$ Then
$
u\wedge v =0   \Leftrightarrow u \diamondsuit v.
$
 \end{lemma}

\begin{lemma}\label{cap}
Let  $v \in \mathcal{GT}.$ Then  $F_v\cap F_e\neq\emptyset$ or
$F_{-v}\cap F_e\neq\emptyset.$
 \end{lemma}

\begin{proof}
Let us show that for any $v \in \mathcal{GT}$ either $F_v\cap
F_e\neq\emptyset$ or  $F_{-v}\cap F_e\neq\emptyset.$ Let $f\in
F_v.$ Equality \eqref{yad}
 implies that
$$
f=tg+(1-t)h,
$$
where
 $g\in F_e,\, h\in F_{-e},\, 0\leq t \leq1.$

If  $t=1,$ then $f=g;$ therefore $f=g\in F_{v}\cap F_e.$

If  $t=0,$ then $f=h;$ therefore $-f\in F_{-v}\cap F_e.$

Now let  $0<t<1.$ Since $F_v$ is a face, we have $g, h\in F_v.$
Hence, $F_v\cap F_e\neq\emptyset$ and  $F_{-v}\cap
F_e\neq\emptyset.$

Thus, we conclude that $F_v\cap F_e\neq\emptyset$ or  $F_{-v}\cap
F_e\neq\emptyset.$
\end{proof}

 \begin{lemma}\label{JDM}
For every $u \in \mathcal{GT}$  there exist mutually orthogonal
geometric tripotents $u_1, u_2\in \nabla$ such that $u=u_1-u_2.$
\end{lemma}

\begin{proof}
Put $ u_1=u\wedge e$ and $u_2=(-u)\wedge e.$ Let us show that $
u_1\diamondsuit u_2$ and $u=u_1+u_2. $ First, assume that
$u_1\wedge u_2\neq 0.$ Then there exists an element $x\in Z_1$
such that
$$
\langle u_1, x\rangle= \langle u_2, x\rangle=1.
$$
Since $u_1=u\wedge e,\, u_2=(-u)\wedge e,$ we obtain
$$
\langle u, x\rangle= 1,\, \langle -u, x\rangle=1,
$$
which is a contradiction. Consequently, $u_1\wedge u_1=0.$
Therefore, in view of Lemma \ref{dis} we get  $u_1\diamondsuit
u_2.$

Assume that $v=u-u_1+u_2\neq 0.$ By virtue of Lemma \ref{cap}
 we have $F_v\cap F_e\neq\emptyset$ or $F_{-v}\cap
F_e\neq\emptyset.$ Without loss of generality, we assume that
$F_v\cap F_e\neq\emptyset.$ This implies that there is an element
$x\in Z_1$ such that
$$
\langle v, x\rangle=\langle e, x\rangle=1.
$$
Since $v\leq u,$ we have $\langle u, x\rangle=1.$ Hence  $x\in
F_u\cap F_e,$ i.e.,  $x\in F_{u_1}$ or $\langle u_1, x\rangle=1.$
The condition $u_1\diamondsuit u_2,$ implies that $\langle u_2,
x\rangle=0.$ Thus, we have
$$
\langle v, x\rangle= \langle u, x\rangle- \langle u_1, x\rangle+
\langle u_2, x\rangle =0,
$$
which contradicts the equality  $\langle v, x\rangle=1.$ This
yields the representation $u=u_1-u_2,$ as required.
\end{proof}

On the space $Z$ we introduce the order relation as follows:
\begin{equation}\label{order}
x\geq 0,\, x\in Z \Leftrightarrow \langle u, x\rangle\geq 0,\,
\forall\, u\in  \nabla.
\end{equation}

\begin{lemma}
\label{pos} Let $x\in Z.$ The following conditions are equivalent:
\begin{enumerate}
\item $x\geq 0;$

\item $v_x\in \nabla;$

\item $x\in \mathbb{R}^+F_e.$
\end{enumerate}
\end{lemma}

\begin{proof}
(1)$\Rightarrow$(2). Let $x\geq 0.$ Consider the least geometric
tripotent $v_x$ such that  $\langle v_x, x\rangle=||x||.$
According to Lemma \ref{JDM}, there exist elements   $u_1, u_2\in
\nabla$ such that $v_x=u_1-u_2.$ Assume that $u_2\neq 0.$

Then we have
$$
\langle u_2, x\rangle=\langle u_2, P_2(v_x)(x)\rangle=
$$
$$
=\langle u_2, P_2(u_1)(x)\rangle- \langle u_2, P_2(u_2)(x)\rangle=
$$
$$
=- \langle u_2, P_2(u_2)(x)\rangle= - \langle P_2(u_2)^\ast (u_2),
x\rangle = - \langle u_2, x\rangle.
$$
This implies that $\langle u_2, x\rangle=0,$ whence  $\langle v_x,
x\rangle=\langle u_1, x\rangle.$ Since $u_1\leq v_x$ and  $v_x$ is
the least tripotent such that $\langle v_x, x\rangle=||x||,$ we
obtain the equality $u_1=v_x.$ Therefore,  $u_2=0$ and
$v_x=u_1\in \nabla.$

(2)$\Rightarrow$(3). Assume that $v_x\in \nabla.$ Since  $v_x\leq
e,$ we get $\frac{\textstyle x}{\textstyle ||x||}\in
F_{v_x}\subset F_e.$ In other words, we have  $x\in
\mathbb{R}^+F_e.$

(3)$\Rightarrow$(1). Finally, let $x\in \mathbb{R}^+F_e.$ Then
$x=\alpha y$ for some $y\in F_e,$ $\alpha\geq 0.$ Therefore, for
each $u\in \nabla$ we have
$$
\langle u, x\rangle=\langle u, \alpha y\rangle= \alpha\langle u,
y\rangle= \alpha ||P_2(u)(y)||\geq 0,
$$
i.e.,  $\langle u, x\rangle\geq 0.$ This means that $x\geq 0.$
\end{proof}

\begin{lemma}
\label{orsp} $Z$  is a partially ordered vector space, i.e., the
following properties are fulfilled:
\begin{enumerate}
\item $x\leq x;$

\item $x\leq y,\,  y\leq    z \Rightarrow  x\leq z;$

\item $x\leq y,\,  y\leq    x \Rightarrow  x=y;$

\item $x\leq y\,   \Rightarrow  x+z\leq y+z;$

\item $x\geq 0,\,  \lambda \geq 0 \Rightarrow  \lambda x\geq 0.$
\end{enumerate}
\end{lemma}

\begin{proof}
Properties (1),(2), (4) and (5) are trivial. Let us prove property
(3). Let  $x\leq y$ and $y\leq    x.$ Then  $\langle v, x-y\rangle
=0$ whenever $v\in \nabla.$ Let  $x\neq y.$ Choose a geometric
tripotent $u\in \mathcal{GT}$ such that
$$
\langle u, x-y\rangle =||x-y||.
$$
By virtue of Lemma  \ref{JDM} there exist two elements  $u_1,
u_2\in \nabla$ such that  $u=u_1-u_2.$ Then, we get $\langle u_1,
x-y\rangle =\langle u_2, x-y\rangle =0.$ This implies the equality
$\langle u, x-y\rangle =0,$ which is a contradiction.
Consequently, we get $x=y,$ as required.
\end{proof}

\begin{lemma}
\label{posit} For each $u\in \nabla$ the operator $P_2(u)$ is
positive.
\end{lemma}

\begin{proof}
Let $x\geq 0.$ By virtue of Lemma  \ref{pos} we have the equality
 $x=\alpha  y,$ where  $y\in F_e,\, \alpha\geq 0.$ If
$P_2(u)(y)=0,$ then  $P_2(u)(y)\in \mathbb{R}^+F_e.$ If
$P_0(u)(y)=0,$ then  $P_2(u)(y)=y\in \mathbb{R}^+F_e.$

Now let $P_2(u)(y)\neq 0$ and  $P_0(u)(y)\neq 0.$ Since
$P_1(u)=0,$ we get
$$
||P_2(u)(y)||+||P_0(u)(y)||=||P_2(u)(y)+P_0(u)(y)||=||y||=1,
$$
therefore
$$
||P_2(u)(y)||\frac{P_2(u)(y)}{||P_2(u)(y)||}+
||P_0(u)(y)||\frac{P_0(u)(y)}{||P_0(u)(y)||}=y\in F_e.
$$
Since $F_e$ is a face, we have the inclusion  $$
\frac{P_2(u)(y)}{||P_2(u)(y)||}\in F_e.
$$
Again using Lemma  \ref{pos}, we obtain $P_2(u)(x)=\alpha
P_2(u)(y)\geq 0.$
\end{proof}

\begin{lemma}
\label{posit} For every $x\in Z$  there exist mutually orthogonal
geometric tripotents
  $u_+, u_-\in \nabla$ such
that $u_++u_-=e$ and
$$
x_+=P_2(u_+)(x)\geq 0,\, x_-=P_2(u_-)(x)\leq 0.
$$
\end{lemma}

\begin{proof}
Choose the least geometric tripotent $v_x\in \mathcal{GT}$ such
that $\langle v_x, x\rangle =||x||.$ According to Lemma \ref{JDM},
there exist  $u_1, u_2\in \nabla$ such that  $v_x=u_1-u_2.$ Put $
u_+=u_1$ and $u_-=-u_2+v_x^\perp. $ Then
\begin{align*}
||x|| & = \langle v_x, P_2(v_x)(x)\rangle= \\
& = \langle u_1, P_2(u_1)(x)\rangle- \langle u_2,
P_2(u_2)(x)\rangle\leq  ||P_2(u_1)(x)||+||P_2(u_2)(x)||=
\\ & = ||P_2(u_1)(x)+P_2(u_2)(x)||=||P_2(v_x)(x)||=||x||.
\end{align*}
This implies the following equalities:
\begin{equation}\label{part}
\langle u_1, P_2(u_1)(x)\rangle= ||P_2(u_1)(x)||,\, \langle -u_2,
P_2(u_2)(x)\rangle= ||P_2(u_2)(x)||.
\end{equation}
Since
$$
P_2(v_x^\perp)(x)= P_0(v_x)(x)=0,$$ from \eqref{part} it follows
tha $u_1$ is the least geometric tripotent such that
$$
\langle u_1, P_2(u_+)(x)\rangle= ||P_2(u_+)(x)||,
$$ and  $-u_2$ is the least geometric tripotent such that
$$
\langle -u_2, P_2(u_-)(x)\rangle= ||P_2(u_-)(x)||.
$$
Consequently, by virtue of Lemma  \ref{pos} we have $
P_2(u_+)(x)\geq 0$  and $P_2(u_-)(x)\leq 0.
$
\end{proof}

Note that the condition $u_+\,\diamondsuit \,u_-$ implies
$x_+\,\diamondsuit \,x_-.$

\begin{lemma}
\label{latt}
  $Z$ is a vector lattice, i.e., for all
  $x, y\in Z$ there exist
$
x\vee y,\, x\wedge y\in Z.
$
\end{lemma}

\begin{proof}
By Lemma \ref{posit} there exist mutually orthogonal geometric
  $u_+, u_-\in \nabla$ с $u_++u_-=e$
such that
$$
P_2(u_+)(x-y)\geq 0,\,
$$
$$
P_2(u_-)(x-y)\leq 0.
$$
Then
\begin{equation}\label{max}
x\vee y=P_2(u_+)(x)+P_2(u_-)(y),
\end{equation}
\begin{equation}\label{min}
x\wedge y=P_2(u_+)(y)+P_2(u_-)(x).
\end{equation}

Really, we have
$$
x\vee y-x=P_2(u_+)(x)+P_2(u_-)(y)-x=
$$
$$
=P_2(u_+)(x-x)+P_2(u_-)(y-x)\geq 0
$$
and
$$
x\vee y-y=P_2(u_+)(x)+P_2(u_-)(y)-y=
$$
$$
=P_2(u_+)(x-y)+P_2(u_-)(y-y)\geq 0.
$$

Assume that $x, y\leq z,$ where $z\in Z.$ Then
$$
z-x\vee y=z-P_2(u_+)(x)-P_2(u_-)(y)=
$$
$$
=P_2(u_+)(z-x)+P_2(u_-)(z-y)\geq 0.
$$
This means that
$$
x\vee y=P_2(u_+)(x)+P_2(u_-)(y).
$$
One can prove \eqref{min} in just the same way.
\end{proof}

\begin{lemma}
\label{nor} Let  $x\in Z$ and $x\geq 0.$ Then
$$
||x||=\langle e, x\rangle.
$$
\end{lemma}

\begin{proof}
Choose the least geometric tripotent $v_x\in \mathcal{GT}$ such
that $v_x(x)=||x||.$ In accordance with Lemma  \ref{orsp} we have
 $v_x\in\nabla.$ Since  $v_x\leq e,$ we get
 $\langle v_x, x\rangle\leq \langle e, x\rangle.$
Therefore,
$$
||x||=\langle v_x, x\rangle\leq \langle e, x\rangle \leq ||x||.
$$
Thus, we obtain the equality $||x||=\langle e, x\rangle.$
 \end{proof}

Recall that a Banach lattice   $X$  is called an abstract
$L$-space, if
$$
||x+y||=||x||+||y||
$$
whenever $x, y\in Z, x, y\geq 0, x\wedge y =0$  (see  \cite[P.
14]{Lin} and  \cite{kak}).

\begin{lemma}
\label{main}
  $Z$  is an abstract  $L$-space.
\end{lemma}

\begin{proof} First, let us prove that
$$
0\leq x\leq y \Rightarrow ||x||\leq ||y||;
$$
$$
|||x|||=||x||.
$$
Let  $0\leq x\leq y.$ Then
$$
||x||=\langle e, x\rangle\leq \langle e, y\rangle=||y||,
$$
i.e.,
$$
||x||\leq ||y||.
$$
Further, we have
$$
|||x|||=||x_++x_-||=[x_+\,\diamondsuit \,x_-]=
$$
$$
=||x_+-x_-||=||x||.
$$
Thus, the space $Z$ is a Banach lattice.

Now let  $x, y\geq 0.$ Using Lemma  \ref{nor}, we get
$$
||x+y||=\langle e, x+y\rangle= \langle e, x\rangle+\langle e,
y\rangle= ||x||+||y||.$$ This means that $Z$ is an abstract
$L$-space.
\end{proof}

Now the proof of the theorem follows from Lemma  \ref{main} and
Theorem 1.b.2 in \cite{Lin}.

\end{document}